\documentclass{amsart}
\usepackage[british]{babel}
\usepackage{amsmath,amssymb,amsthm,amscd}
\usepackage{graphicx,subfigure}
\usepackage{color}
\usepackage{pb-diagram}
\usepackage{tabularx}
\usepackage{marginnote}
\usepackage{verbatim}
\usepackage{comment}

\newtheorem{theorem}{Theorem}[section]
\newtheorem{proposition}[theorem]{Proposition}

\newtheorem{lemma}[theorem]{Lemma}

\theoremstyle{definition}

\theoremstyle{remark}
\newtheorem{remark}[theorem]{Remark}

\numberwithin{equation}{section}

\def\AA{{\mathbb A}}
\newcommand{\RR}{{\mathbb R}}

\newcommand{\TT}{{\mathbb T}}
\newcommand{\ZZ}{{\mathbb Z}}

\newcommand{\QQ}{{\mathbb Q}}

\newcommand{\tPhi}{\tilde{\Phi}}

\newcommand{\abs}[1]{|{#1}|}

\newcommand{\norm}[1]{\|{#1}\|}

\def\epsilon{\varepsilon}

\def\Dif{ {\mbox{\rm D}} }

\def\df{\dot{f}}       
\def\ddf{\ddot{f}}     

\def\to{t}        
\def\Pto{\bar t}  
\def\vo{v}        
\def\Pvo{\bar v}  
\def\Eo{e}        
\def\PEo{\bar e}  

\DeclareMathOperator{\supp}{supp}
\DeclareMathOperator*{\esssup}{ess\,sup}
\DeclareMathOperator*{\essinf}{ess\,inf}

\allowdisplaybreaks

\begin{document}

\title{Diffusion and chaos in a bouncing ball model}

\author{Stefano Mar\`o}
\address{Dipartimento di Matematica,  Universit\`a di Pisa, Largo Bruno Pontecorvo 5, 56127 Pisa, Italy}
\email{stefano.maro@unipi.it}

\thanks{This work has been supported by the PRIN Project "Regular and stochastic behaviour in dynamical systems", funded by the Ministry of Education, University and Scientific Research of Italy with the identification no. 2017S35EHN}

\begin{abstract}
We consider the vertical motion of a free falling ball bouncing elastically on a racket moving in the vertical direction according to a regular periodic function $f$. We give a sufficient condition on the second derivative of $f$ giving motions with arbitrarily large amplitude and chaotic dynamics in the sense of positive entropy. We get the results by breaking many invariant curves of the corresponding map using converse KAM techniques.  
\end{abstract}
\maketitle


\section{Introduction}
\noindent The vertical dynamics of a free falling ball on a moving racket is considered. The racket is supposed to
move periodically in the vertical direction according to a regular periodic function $f(t)$ and
the ball is reflected according to the law of elastic bouncing when hitting the
racket. The only force acting on the ball is the gravity $g$. Moreover, the mass of the racket is assumed to be large with respect to the mass of the
ball so that the impacts do not affect the motion of the racket.

\noindent This model has inspired many authors as it represents a simple model exhibiting complex dynamics; see for example \cite{dolgo,kunzeortega2,maro2,maro3,maro4,pust,ruiztorres}.

We will be concerned with the existence of motions in which the the velocity of the ball could pass from a small value to an arbitrarily large one (see the statement of Theorem \ref{teo_main} for more details). We will call this kind of motions diffusive, being freely inspired by the well known phenomenon of Arnold diffusion in nearly integrable Hamiltonian systems.\\
A first example of diffusive motions is given by unbounded motions, in which the velocity tends to infinity. On this line, Pustylnikov \cite{pust} showed that 
if
\begin{equation}
  \label{pustcond}
\df(t_0)\geq\frac{g}{2}
\end{equation}
at some point $t_0$, then there exist motions that gain velocity at every bounce escaping to infinity.\\
Large values of the first derivative are not necessary to have unbounded motions. Actually, in \cite{maro4} it was proven that one can construct periodic functions $f$ with arbitrarily small first derivative for which there exist motions that gain velocity at every $N$ bounces for some $N$, escaping to infinity.\\
These results rely on the existence of some resonance, represented by different bounces when the racket has the same height and moves upwards. Moreover, the initial velocity has to be large enough.

In this paper we show the existence of diffusive motions via an indirect way. More precisely, we will get the following condition, depending only on the second derivative of $f$. Denote $m = \min\ddf$ and $M=\max\ddf$ then, if
\begin{equation}
  \label{main}
m <- \frac{g}{   1+\sqrt{1 +\frac{g}{M}}  }
\end{equation}
then there exist diffusive motions for initial velocities sufficiently large. 

The possible motions of the ball can be described by the orbits of a map $\Psi$, that we call Tennis Map for clear reasons. The Tennis Map turns out to be  exact symplectic and twist when defined on the cylinder with coordinates $(t,e)$, time of bouncing and energy just after the bounce. For this class of maps, it comes from a theorem of Birkhoff that the destruction of rotational invariant curves implies the existence of diffusive orbits. From this, the main idea for the proof of the results is that condition \eqref{main} shall imply that there are no rotational invariant curves for large values of $e$.

The study of necessary conditions for the existence of invariant curves for symplectic twist maps goes back to Birkhoff who proved that a rotational invariant curve is the graph of a (periodic) Lipschitz function. It means that the oscillations of the tangent vectors to an orbit are controlled by the Lipschitz constant $L$. Moreover, a sharp estimate on $L$ would give a more stringent criterion. In the last decades many results have been proven in this direction and extended to higher dimensions giving rise to the so called ``Converse KAM'' theory \cite{percivalmk,mkmeiss,haro}.    

These criteria have a variational characterization. Orbits of exact symplectic twist maps correspond to stationary points of an action and the ones on invariant curves are action-minimizing. As a consequence, the second variation of the action must be positive on orbits on invariant curves. MacKay and Percival \cite{percivalmk} showed that this criterion is equivalent to the control of the oscillations taking the Lipschitz constant $L=\infty$. Using the same technique we show that an improvement of the estimate on $L$ gives a strictly positive lower bound for the second variation of the action. The value in \eqref{main} will come from an application of this last result. It is easy to note that requiring only the positiveness of the second variation we would get a result on existence of diffusive orbits under the condition
\[
m<-\frac{g}{2}
\]
that is stronger than \eqref{main}.

Finally, we also note that the destruction of invariant curves is strictly correlated with the presence of chaotic dynamics. This fact holds for exact symplectic twist maps of the cylinder and is based on Aubry-Mather theory. More precisely, the chaotic character of the dynamics in a Birkhoff region of instability is described in \cite{angenent1,angenent2,matherams}. However, our case is slightly different: we obtained the destruction of invariant curves in an open and unbounded region of the cylinder, while a Birkhoff region of instability is a compact region between two invariant curves. 

In \cite{maro2} we proved that the results in \cite{angenent1,angenent2} can be extended to the present case, if condition \eqref{pustcond} together with a similar one giving decelerating orbits are satisfied. More precisely, we got semiconjugation with the Bernoulli shift and therefore positive topological entropy.\\
We will show that condition \eqref{main} alone also gives positive entropy for our map. This will follow from an application of a result of Forni, giving the existence of invariant measures with positive entropy supported in the gaps of the Aubry-Mather sets. 

The paper is organized as follows. In Section \ref{sec:statement} we introduce the Tennis Map $\Psi$ and describe our main results. In Section \ref{sec:exact} we introduce exact symplectic twist maps and describe two necessary conditions for the existence of invariant curves. In Section \ref{sec:applications} we apply the results from Section \ref{sec:exact} to the Tennis Map proving the non existence of invariant curves. In Section \ref{sec:chaos} we discuss the proof of our main theorem as a consequence of the non existence of invariant curves. Conclusions are drawn in Section \ref{sec:conclusion}.

\section{Statement of the problem and main result}\label{sec:statement}

Consider the problem of the motion of a bouncing ball on a vertically
moving racket. We assume that the impacts do not affect the racket  whose vertical position is described by a
$1$-periodic $C^3$ function $f:\RR\rightarrow\RR$. To get the equations of motion, we put ourselves 
in
an inertial frame, denoting by $(\to,w)$ the time of impact and the corresponding velocity just after the bounce, and by
$(\Pto,\bar w)$ the corresponding values at the subsequent bounce.
From
the free falling condition we have
\begin{equation}\label{timeeq}
f(t) + w(\Pto-\to) - \frac{g}{2}(\Pto -\to)^2 = f(\Pto) \,,
\end{equation}
where $g$ stands for the standard acceleration due to gravity.
Noting that the velocity just before the impact at time $\Pto$ is $w-g(\Pto-\to)$, using
the elastic impact condition and recalling that the racket is not affected by the ball, we obtain
\begin{equation}
  \label{veleq}
\bar{w}+w-g(\Pto-\to) = 2\dot{f}(\Pto)\,,
\end{equation}
where $\dot{}$ stands for the derivative with respect to time. From conditions \eqref{timeeq},\eqref{veleq} we can define a bouncing motion given an initial condition $(t,w)$ in the following way. If $w\leq\df(t)$ then we set $\bar{t}=t$ and $\bar{w}=w$. If $w>\df(t)$, we choose $\bar{t}$ to be the smallest solution $\bar{t}\geq t$ of \eqref{timeeq}. Bolzano theorem gives the existence of such solution considering
\[
F_t(\bar{t})=f(t)-f(\Pto) + w(\Pto-\to) - \frac{g}{2}(\Pto -\to)^2 
\]
and noting that $F_t(\bar{t})<0$ for $\Pto-\to$ large and $F_t(\bar{t})>0$ for $\Pto-\to\rightarrow 0^+$. For this value of $\bar{t}$, condition \eqref{veleq} gives the updated velocity $\bar{w}$.

For $\Pto-\to>0$, we introduce the notation
\[
f[\to,\Pto]=\frac{f(\Pto)-f(\to)}{\Pto-\to},
\]
and write
\begin{equation}
  \label{w1}
\Pto = \to + \frac 2g w -\frac 2g f[\to,\Pto]\,,
\end{equation}
that also gives
\begin{equation}
  \label{w2}
\bar{w}= w -2f
[\to,\Pto]
+ 2\dot{f}(\Pto).
\end{equation}
Now we change to the moving frame attached to the racket, where the velocity after the impact is expressed as $v=w-\dot{f}(t)$,
and we get
the equations
\begin{equation}\label{eq:unb}
  \left\{
  \begin{split}
\Pto = {} & \to + \frac 2g \vo-\frac 2g f[\to,\Pto]+\frac 2g \df(\to)\textcolor{blue}{\,}
\\
\Pvo = {} & \vo - 2f[\to,\Pto] + \df (\Pto)+\df(\to)\textcolor{blue}{\,.}
\end{split}
\right.
\end{equation}
 By the periodicity of the
function $f$, the coordinate $t$ can be seen as an angle. Hence, equations \eqref{eq:unb} define formally a map
\[
\begin{array}{rcl}
\Psi: 
\AA & \longrightarrow & \AA \\
(\to,\vo) & \longmapsto & (\Pto, \Pvo),
\end{array}
\]
where we denoted $\AA = \TT\times \RR$ with $\TT = \RR/\ZZ$.
This
is the formulation considered by Kunze and Ortega
\cite{kunzeortega2}. Another approach was considered by Pustylnikov in
\cite{pust} and leads to a map that is equivalent to
\eqref{eq:unb}, see \cite{maro3}.
Noting that $w>\df(t)$ if and only if $v>0$, we can define a bouncing motion as before and denote it as a sequence $(t_n,v_n)_{n\in\ZZ^+}$ with $\ZZ^+=\{n\in\ZZ \::\: n\geq 0\}$ such that $(t_n,v_n)\in \TT\times [0,+\infty)$ for every $n\in\ZZ^+$.

We are going to show the following result
\begin{theorem}\label{teo_main}
  Suppose that $f\in C^3(\TT)$ is such that
\begin{equation*}
m <- \frac{g}{   1+\sqrt{1 +\frac{g}{M}}  },
\end{equation*}
where  $m=\min \ddf, M=\max\ddf$. Then
\begin{itemize}
\item for every $A>0$ there exists a bouncing motion $(t_n,v_n)_{n\in\ZZ^+}$ such that \[
 \sup_{n\in\ZZ^+}v_n- \inf_{n\in\ZZ^+}v_n >A,
  \]

\item there exist many compact $\Psi$-invariant subsets of $\AA$  on which $\Psi$ has positive topological entropy.
  \end{itemize}

  \end{theorem}

\section{Exact symplectic twist maps and destruction of invariant curves}\label{sec:exact}
In this section we consider a (possibly unbounded) strip of the cylinder and denote it by $\Sigma=\TT\times(a,b)$ with $-\infty\leq a<b\leq +\infty$. Let  $S:\Sigma \rightarrow \AA$, be a $C^2$-embedding and denote $S(x,y)=(\bar{x},\bar{y})$ and $S^n(x,y)=(x_n,y_n)$. We suppose that $S$ is exact symplectic and twist.
%
 The exact symplectic condition requires the existence of a $C^2$ function $V:\Sigma\rightarrow \RR$ such that
\[
\bar{y} d \bar{x} -y dx = dV(x,y) \quad\mbox{in }\Sigma,
\]
and the (positive) twist condition reads
\[
\frac{\partial \bar{x}}{\partial y} >0 \quad\mbox{in }\Sigma. 
\]
A negative twist condition would give analogous results. Moreover, the exact symplectic condition implies that $S$ is orientation preserving and preserves the two-form $dy\wedge dx$.  
For this class of maps, the following result is well known \cite{bangert,matherforni}. 
\begin{proposition}
  \label{prop_h}
  There exist a domain $\Omega\subset \RR^2$ and a $C^2$ function $h:\Omega\rightarrow\RR$ such that
  \begin{itemize}
  \item $h(x+1,\bar{x}+1) = h(x,\bar{x})$ in $\Omega$,
  \item $h_{12}(x,\bar{x}) <0 $ in $\Omega$,
  \item for $(x,y)\in\Sigma$ we have $S(x,y) = (\bar{x},\bar{y})$ if and only if
    \[
    \left\{
    \begin{split}
      h_1(x,\bar{x})&=-y \\
      h_2(x,\bar{x})&=\bar{y}.
    \end{split}
    \right.
\]
  \end{itemize}
\end{proposition}
\begin{remark}
Here we denoted the partial derivative of $h$ w.r.t the $i$-th variable by $h_i$. We will use this notation throughout the paper.
\end{remark}  
\begin{remark}
  The domain $\Omega$ can be defined in the following way (see \cite{ortegakunze}):
  \begin{equation}
    \label{defomega}
    \Omega = \left\{ (x,\bar{x})\in\RR^2 \: :\: \bar{x}(x,a)\leq \bar{x}\leq \bar{x}(x,b)   \right\}.
    \end{equation}
  The condition $h_{12}(x,\bar{x}) <0 $ is related with the twist condition. Actually, the twist implies that we can write $y = y(x,\bar{x})$ and one gets that
  \[
h_{12}(x,\bar{x}) = -\left(\frac{\partial \bar{x}}{\partial y}(x,y(x,\bar{x}))\right)^{-1}.
  \]
  \end{remark}
The function $h$ is called generating function and gives an equivalent implicit definition of the diffeomorphism $S$. From this proposition one has that a sequence $(x_n,y_n)_{n\in\ZZ}$ such that $(x_n,y_n)\in\Sigma$ for every $n\in\ZZ$ is an orbit of $S$ if and only if for every $n\in\ZZ$, $(x_n,x_{n+1})\in \Omega$ and 
\begin{align}
\label{acca}  &h_2(x_{n-1},x_n) + h_1(x_n,x_{n+1})=0,  \\
  &y_n=-h_1(x_{n},x_{n+1})\nonumber.    
\end{align}

From now on we will consider the case $\Sigma =\AA$ and suppose that $S$ preserves the ends of the cylinder that is
\[
\bar{y} \rightarrow \pm\infty \quad\mbox{as }y\rightarrow \pm\infty \mbox{ uniformly in } x,
\]
and twists each ends infinitely that is
\[
\bar{x}-x \rightarrow \pm\infty \quad\mbox{as }y\rightarrow \pm\infty \mbox{ uniformly in } x.
\]
Here, with some abuse of notation, we still denoted $\bar{x}(x,y)$ the first component of the lift of $S$ to the universal cover $\RR^2$ of $\AA$. In particular, $x\in\RR$ and $\bar{x}(x+1,y)=\bar{x}(x,y)+1$.\\ 
In this way, the generating function is defined in $\Omega=\RR^2$.
This allows to give a variational characterization of the orbits of $S$ defining the action
\[
H_{hk}(x_h,\dots,x_k)=\sum_{n=h}^{k-1} h(x_n, x_{n+1})
\]
and seeing that solutions of \eqref{acca} (and hence orbits of $S$) are in 1-1 correspondence with stationary points of $H_{hk}$
with respect to variations fixing the endpoints $x_h,x_k$.

We will be interested in action minimizing orbits, i.e. orbits $(x_n,y_n)_{n\in\ZZ}$ of $S$ such that for every pair of integers $h<k$ and for every sequence of real numbers $(x^*_n)_{h\leq n\leq k}$ such that $x_h^*=x_h$ and $x_k^*=x_k$ it holds
\begin{equation*}
 H_{hk}(x_h,\dots,x_k)    \leq H_{hk}(x^*_h,\dots,x^*_k).
    \end{equation*}

In this framework, we will be concerned with necessary conditions for the existence of invariant curves for $S$. More precisely, an invariant curve will be a curve $\Gamma\subset\Sigma$ homotopic to $\{(x,y)\in\AA \: : \:y=k, \mbox{ for some }k\in\RR  \}$ and such that $S(\Gamma) = \Gamma$.

Let us start recalling the following well known result by Birkhoff (for a proof see \cite{herman,mather_no,percivalmk}).
\begin{theorem}
  \label{birk}
Every invariant curve $\Gamma$ of a symplectic twist diffeomorphism of $\AA$ that preserves and twists infinitely the ends is the graph of a Lipschitz function $P:\TT\rightarrow \RR$, i.e. $\Gamma =\{(x,P(x))\in\AA \: :\: x\in\TT \}$. Moreover, if there exist $y^-<y^+$ such that each orbit $(x_n,y_n)_{n\in\ZZ}$ with $y_0<y^-$ satisfies $y_n<y^+$ for every $n\in\ZZ$, then there exists an invariant curve $\Gamma\subset \TT\times (y^-,y^+)$.
\end{theorem}

The necessary conditions that we are going to give, will be concerned with the properties of the orbits on the invariant curve itself. For this purpose, we state here a consequence of Theorem \ref{birk}.

\begin{lemma}
  \label{lemfi}
  For every invariant curve $\Gamma$ of $S$ there exists an increasing bi-Lipschitz homeomorphism $\varphi:\RR\rightarrow\RR$ such that $\varphi(x+1)=\varphi(x)+1$ and 
  \begin{equation}
    \label{dEL}
  h_2(\varphi^{-1}(x),x)+h_1(x,\varphi(x)) =0,
  \end{equation}
for every $x\in\RR$.  Moreover, for every $x,x_*\in\RR$
  \[
0<\essinf_\TT \varphi'\leq \frac{\varphi(x)-\varphi(x_*)}{x-x_*}\leq\esssup_\TT \varphi' <\infty.
  \]
\end{lemma}
\begin{proof}
 By Birkhoff theorem there exists a Lipschitz function $P: \TT\rightarrow\RR$ such that
  \[
  \Gamma =\left\{(x,P(x)) \: :\: x\in\TT   \right\}.
  \]
From now on, with some abuse of notation, we will still denote with the same letter $S$ or $P$ the corresponding lifts. Since $P$ is Lipschitz and $\Gamma$ is invariant, the map $\varphi:\RR\rightarrow\RR$ defined as
  \begin{equation}\label{defi}
  \varphi(x) = \pi_1\circ S(x,P(x))
  \end{equation}
  is an increasing Lipschitz homeomorphism such that $\varphi(x+1)=\varphi(x)+1$ and
  \[
S(x,P(x) ) = (\varphi(x), P(\varphi(x)))\quad  \mbox{for all }x\in\RR.  
  \]
  Then, from \eqref{acca}, 
  \begin{equation*}
  h_2(\varphi^{-1}(x),x)+h_1(x,\varphi(x)) =0.
   \end{equation*} 
  Since $\varphi$ is Lipschitz-continuous it admits a derivative $\varphi'$ defined almost everywhere and its Lipschitz constant turns out to be equal to $\esssup_\TT \varphi'$. On the other hand, $\varphi^{-1}(x)= \pi_1\circ S^{-1}(x,P(x))$ is also Lipschitz and $(\varphi^{-1})'=\frac{1}{\varphi'\circ\varphi^{-1}}$ so that $\essinf_\TT \varphi'>0$. Summing up, $\varphi$ is a bi-Lipschitz homeomorphism such that, for every $x,x_*\in\TT$
  \[
0<\essinf_\TT \varphi'\leq \frac{\varphi(x)-\varphi(x_*)}{x-x_*}\leq\esssup_\TT \varphi' <\infty.
  \]
\end{proof}  

Given an invariant curve $\Gamma$ of $S$, we consider the corresponding homeomorphism $\varphi$ from Lemma \ref{lemfi} and introduce the functions $a,b : \RR\rightarrow\RR$ defined as
\begin{equation}
  \label{defab}
 a(x) = h_{22}(\varphi^{-1}(x),x)+h_{11}(x,\varphi(x)), \quad   b(x)= -h_{12}(\varphi^{-1}(x),x),
\end{equation}
where $h$ is the generating function introduced in Proposition \ref{prop_h}. Note that both functions are continuous and $1$-periodic and $b$ is strictly positive.

The following result gives a first necessary condition for the existence of invariant curves. Even if it was proven in \cite{mather_glancing} we give here the proof since we will need it later. 

\begin{theorem}
  \label{teomin}
  If $\Gamma$ is an invariant curve of $S$, then
  \[
a(x) >0 \quad \mbox{for all }x\in\RR. 
  \]
\end{theorem}
\begin{proof}
  From an application of Lemma \ref{lemfi} we get a bi-Lipschitz homeomorphism $\varphi$ satisfying \eqref{dEL}. This last equation can be differentiated in almost every point leading to
    \begin{equation*}
    \label{deffi}
h_{22}(\varphi^{-1}(x),x)+h_{11}(x,\varphi(x)) = -h_{12}(x,\varphi(x))\varphi'(x)-h_{21}(\varphi^{-1}(x),x)\frac{1}{\varphi'(\varphi^{-1}(x))} \quad\mbox{a.e. }x\in\RR, 
    \end{equation*}
    that using \eqref{defab} and Proposition \ref{prop_h} becomes
  \begin{equation}
    \label{deffi}
a(x) = b(\varphi(x))\varphi'(x)+\frac{b(x)}{\varphi'(\varphi^{-1}(x))} \quad\mbox{a.e. }x\in\RR, 
    \end{equation}
and $b(x)>0$ for every $x\in\RR$. Hence, from Lemma \ref{lemfi} we get that 
\[
a(x) > b(\varphi(x))\essinf_\TT \varphi'+\frac{b(x)}{\esssup_\TT \varphi'} >0  \quad\mbox{a.e. }x\in\RR.
\]
The last inequality holds for all $x\in\RR$, hence by continuity we get that for every $x\in\RR$,
\begin{equation}
  \label{fine_inizio}
a(x) \geq b(\varphi(x))\essinf_\TT \varphi'+\frac{b(x)}{\esssup_\TT \varphi'} >0.  
\end{equation}

\end{proof}
  

Following the ideas in \cite{percivalmk,mkmeiss} Theorem \ref{teomin} can be improved.

\begin{theorem}\label{teo_nec}
Let $\Gamma$ be an invariant curve of $S$ and consider four positive constants $B^\pm,C^\pm$ such that 

  \begin{equation*}
    \begin{split}
       B^+ &\geq \max_{\RR} \frac{a(x)}{b(\varphi(x))}, \quad C^+ \leq \min_{\RR}\frac{b(x)}{b(\varphi(x))}, \\
      B^- &\geq \max_{\RR} \frac{a(x)}{b(x)}, \quad C^- \leq \min_{\RR}\frac{b(\varphi(x))}{b(x)},
    \end{split}
  \end{equation*}
  where the functions $a,b$ have been defined in \eqref{defab} and the function $\varphi$ comes from Lemma \ref{lemfi}. Suppose that 
  \begin{equation}
    \label{cond_pos}
    (B^\pm)^2-4C^\pm>0.
  \end{equation}
  Then, for every $x\in\RR$,
  \begin{equation*}
  a(x) \geq b(\varphi(x))D^- +\frac{b(x)}{D^+},
  \end{equation*} 
  where
  \[
D^+ = \frac{1}{2}\left(B^+ +\sqrt{(B^+)^2 -4C^+}  \right), \quad D^- = \frac{1}{2C^-}\left(B^- -\sqrt{(B^-)^2 -4C^-}  \right).
  \]

\end{theorem}

%

\begin{remark}
  \label{rem_C}
  Note that by Theorem \ref{teomin} and the twist condition the periodic and continuous functions $a(x)$ and $b(x)$ are strictly positive. Hence, the continuous functions
  \[
  \frac{a(x)}{b(\varphi(x))}, \quad \frac{a(x)}{b(x)}, \quad \frac{b(x)}{b(\varphi(x))}, \quad \frac{b(\varphi(x))}{b(x)}
  \]
  are periodic and attain a positive minimum on $\RR$. This implies that the constants $B^\pm$ are already positive and that the constants $C^\pm$ satisfying the desired estimates actually exist. 
  \end{remark}

\begin{remark}
Theorem \ref{teo_nec} yields Theorem \ref{teomin} in the limit case $D^-\rightarrow 0$, $D^+\rightarrow +\infty$. In the other cases, Theorem \ref{teo_nec} gives a stronger necessary condition since $\min_\RR b(x)>0$ and $D^-,D^+>0$.  
\end{remark}

\begin{remark}
  Let us consider the standard map
\begin{equation*}
\left\{
\begin{split}
&\bar{x} = x + y - k\sin{x}  \\
&\bar{y} = y - k\sin{x},
\end{split}
\right. 
\end{equation*}
where $k>0$. One has that $h(x,\bar{x}) =\frac{1}{2}(\bar{x}-x)^2+k\cos{x}$, so that, given an invariant curve, $a(x) = 2-k \cos x$ and $b(x) =1$ for every $x\in\RR$. \\
To apply Theorem \ref{teomin} we note that $a(0) = 2-k <0 $ is satisfied for $k>2$. It means that there cannot exist invariant curves crossing the vertical line $x=0$ so there are not invariant curves at all.
\\
To apply Theorem \ref{teo_nec} we first note that we can take
\[
B^\pm = 2+k, \quad C^\pm=1,
\]
so that
\[
D^+ = 1+\frac{k}{2} + \sqrt{\frac{k^2}{4}+k }, \quad D^- = 1+\frac{k}{2} - \sqrt{\frac{k^2}{4}+k }=\frac{1}{D^+}.  
\]
We need to find $x\in\RR$ such that
\[
2-k\cos x < 2\left( 1+\frac{k}{2} - \sqrt{\frac{k^2}{4}+k }  \right). 
\]
As before choosing $x=0$ we have that
\[
2-k < 2\left( 1+\frac{k}{2} - \sqrt{\frac{k^2}{4}+k }  \right)
\]
is satisfied if $k>4/3$. As a consequence, for $k>4/3$ there are no invariant curves.
In this way, we recover the famous result by Mather \cite{mather_no}.

  \end{remark}

The rest of the section is dedicated to the proof of Theorem \ref{teo_nec}.

\begin{proof}[Proof of Theorem \ref{teo_nec}] Let us fix an invariant curve $\Gamma$ and consider the corresponding bi-Lipschitz homeomorphism $\varphi$ coming from Lemma \ref{lemfi}. Recalling equation \eqref{fine_inizio} we have
  \begin{equation}
    \label{fine}
a(x) \geq b(\varphi(x))\essinf_\TT \varphi'+ b(x)\frac{1}{\esssup_\TT \varphi'}, \qquad \mbox{for every }x\in\RR.
  \end{equation}
  The purpose of the proof is to find explicitly two positive constants $D^-$ and $D^+$ such that $D^-\leq\essinf_\TT \varphi' $ and $D^+\geq\esssup_\TT \varphi' $.\\
  Let us denote by $\mathcal{N}_0$ the subset of $\TT$ of null measure on which $\varphi$ or $\varphi^{-1}$ are not differentiable. Hence the set $\mathcal{N} =\bigcup_{n\in\ZZ}\varphi^n(\mathcal{N}_0)$ has zero measure, and is such that every point $x^0\in\TT\setminus\mathcal{N}$ has the following property: both $\varphi$ and $\varphi^{-1}$ are differentiable along its orbit $(x^0_n)_{n\in\ZZ}$, where $(x^0_n)=\varphi^n(x^0)$.\\
  Consider $x^0\in\TT\setminus\mathcal{N}$ and denote
  $\varphi_n' = \varphi'(x^0_n)$. Since $\varphi$ satisfies equation \eqref{deffi}, recalling also its definition \eqref{defi}, we get for every $n$
  \begin{equation}
    \label{phiineq}
\varphi_n'=\frac{a(x_n^0)}{b(x_{n+1}^0)}-\frac{b(x_{n}^0)}{b(x_{n+1}^0)}\frac{1}{\varphi_{n-1}'} \leq B^+ -\frac{C^+}{\varphi_{n-1}'},
  \end{equation}
  where the last inequality comes from the definitions of $B^+,C^+$.
  Hence, from \eqref{phiineq}
  \begin{align*}
\varphi'(x^0) \leq B^+ -\frac{C^+}{\varphi_{-1}'}     \leq B^+ - \frac{C^+}{B^+ -\frac{C^+}{\varphi_{-2}'}}\leq B^+ - \frac{C^+}{B^+ -\frac{C^+}{B^+ -\frac{C^+}{\varphi_{-3}'}}}\leq \dots
  \end{align*}
This gives an expression for the upper bound $D^+$ as a formal continued fraction. This procedure can be made rigorous, since from the hypothesis (see also Remark \ref{rem_C}) $B^+>0$, $C^+>0$ and $(B^+)^2-4C^+>0$. This implies that the continued fraction expansion converges to the largest solution of the quadratic equation
  \[
d^2-B^+d+C^+ =0
  \]
  that is
  \[
D^+ = \frac{1}{2}\left(B^+ +\sqrt{(B^+)^2 -4C^+}  \right).
  \]
We get $D^-$ in a similar way. From equation \eqref{deffi} we also have, for  $x^0\in\TT\setminus\mathcal{N}$ and for every $n$,
  \begin{equation}
    \label{phiineqm}
\varphi_n'=\left( \frac{a(x_{n+1}^0)}{b(x_{n+1}^0)}-\frac{b(x_{n+2}^0)}{b(x_{n+1}^0)}\varphi_{n+1}'\right)^{-1} \geq \left(B^- -C^-\varphi_{n+1}'\right)^{-1}.
    \end{equation}
  Equation \eqref{phiineqm} gives the estimate of the lower bound $D^-$ in terms of the inverse of a continued fraction:
  \[
\varphi'(x^0)\geq \left(B^- -C^-\varphi_{1}'\right)^{-1} \geq \left(B^- -\frac{C^-}{B^--C^-\varphi_{2}'}\right)^{-1} \geq \left(B^- -\frac{C^-}{B^--\frac{C^-}{B^--C^-\varphi_{3}'}}\right)^{-1}  \geq\dots 
  \]
As before, the continued fraction converges to the  largest solution of the quadratic equation
  \[
d^2-B^-d+C^- =0,
  \]  
  and $D^-$ is the inverse of such root. Finally,
  \[
D^- = \frac{1}{2C^-}\left(B^- -\sqrt{(B^-)^2 -4C^-}  \right).
  \]
  Since neither $D^+$ nor $D^-$ depend on the orbit on $\Gamma$, we can plug them into \eqref{fine} and get the thesis.
\end{proof}

\section{Applications to the Tennis Map}\label{sec:applications}

We apply Theorems \ref{teomin} and \ref{teo_nec} to the Tennis Map $\Psi$ introduced before in \eqref{eq:unb}. The main result of the section is stated and proved in Proposition \ref{propmain}. 
Let us introduce the notation $\RR_{v_*}=\{v\in\RR \: :\: v>v_* \}$ and $\AA_{v_*} = \TT\times \RR_{v_*}$. We will denote the $\sup$ norm by $\norm{\cdot}$ and recall that $f\in C^3(\TT)$.

In Section \ref{sec:statement} we described how to define a bouncing motion for the Tennis Map $\Psi$. However, to apply the theorems, we need some regularity for the map. This is guaranteed by the following Lemma.

\begin{lemma}
  \label{well_def}
There exists $v_*>4\norm{\df}$ such the map $\Psi:\AA_{v_*}\rightarrow \AA$ is a $C^2$ embedding. 
\end{lemma}
\begin{proof}
To prove that the map is well defined and $C^2$ we apply the implicit function theorem.  
To this aim we introduce the $C^{2}$ function $F :\{(\to,\vo,\Pto,\Pvo)\in \RR^2 \times \RR^2 \: :\:\to\neq\Pto   \} \rightarrow \RR^2$
given by

\begin{equation*}
 F(\to,\vo,\Pto,\Pvo):=
\left(
\begin{split}
& \Pto - \to - \frac 2g \vo + \frac 2g f[\to,\Pto]-\frac 2g \df (\to) \\
& \Pvo - \vo + 2f[\to,\Pto] - \df(\Pto)-\df(\to)
\end{split}
\right),
\end{equation*}
and compute   
\[
\Dif_{\Pto,\Pvo} F(\to,\vo,\Pto,\Pvo)=
\begin{pmatrix}
1 + \frac 2g \partial_{\Pto} f[\to,\Pto] & 0 \\
2\partial_{\Pto} f[\to,\Pto]-\ddf (\Pto) & 1
\end{pmatrix},
\]
where
\[
\partial_{\Pto}f[\to,\Pto]=\frac{\df(\Pto)-f[\to,\Pto]}{\Pto-\to} \,.
\]
Let us now consider a point satisfying
$F(\to_0,\vo_0,\Pto_0,\Pvo_0)=0$, with $(\to_0,\vo_0)\in \AA_{v_{**}}$ and $v_{**}=4\norm{\df}$.
Since $v_0>v_{**}$ we have
\[
\Pto_0 - \to_0 > \frac{2}{g} v_0 - \frac{2}{g}
\abs{f[\to_0,\Pto_0]+\df(\to_0)} > \frac{4}{g} \norm{\df}
\]
and we check the condition of the implicit function theorem:
\[
\det (\Dif_{\Pto,\Pvo} F(\to_0,\vo_0,\Pto_0,\Pvo_0))
=1+\frac{2}{g}\partial_{\Pto}f[\Pto_0,\to_0]>1-
\frac{\frac{4}{g}\norm{\df}_0}{\Pto_0-\to_0}
>0.
\]
Hence, 
we get a $C^{2}$ map $\Psi_0 : U_0 \rightarrow \AA$, defined in a
neighbourhood $U_0 \ni (\to_0,\vo_0)$
such that
$
(\Pto_0,\Pvo_0)=\Psi_0(\to_0,\vo_0)
$
and
$
F(\to,\vo,\Psi_0(\to,\vo))=0
$
for every $(\to,\vo) \in U_0$. Notice that the previous computations
do not depend on the chosen point $(\to_0,\vo_0)$, 
so the local map $\Psi_0$ extends to a global map $\Psi$
(the implicit function theorem can be applied in the full set $\AA_{v_{**}}$).
Moreover, using that $f[\Pto,\to]=f[\Pto+1,t+1]$ and the uniqueness of
the implicit function, we observe that the points 
$(\Pto_0,\Pvo_0)=\Psi(\to_0,\vo_0)$ and $(\Pto_1,\Pvo_1)=\Psi(\to_0+1,\vo_0)$
satisfy $\Pto_0+1=\Pto_1$ and $\Pvo_0=\Pvo_1$.

Moreover, one can easily check that $\Psi$ is a local diffeomorphism since
\[
\det \Dif_{\to,\vo} \Psi (t,v) =-\frac{\det (\Dif_{\to,\vo} F(\to,\vo,\Psi(\to,\vo) ))}{ \det (\Dif_{\Pto,\Pvo} F(\to,\vo,\Psi(\to,\vo) ))} \neq 0 \quad\mbox{on }\AA_{v_{**}}.
\]
To prove that $\Psi$ is an embedding we prove that it is injective in $\AA_{v_*}$ for $v_*$ eventually larger than $v_{**}$. It is convenient to use the variable $w=v+\dot{f}(t)$ and the conjugated map $\tilde{\Psi}$ given by (\ref{w1}-\ref{w2}). Suppose that there exist $(t_1,w_1)$ and $(t_2,w_2)$ such that $\tilde{\Psi}(t_1,w_1)=\tilde{\Psi}(t_2,w_2)=(\bar{t},\bar{w})$. From implicit differentiation of (\ref{w1}-\ref{w2}) we get that for large $w$
\[
\Dif_{\to,w} \tilde{\Psi} (t,w)
\begin{pmatrix}
1 + O(\frac{1}{w})  & \frac{2}{g}+ O(\frac{1}{w}) \\
2\ddf (\Pto) +O(\frac{1}{w}) & 1+\frac{4}{g}\ddf (\Pto)+ O(\frac{1}{w})
\end{pmatrix}.
\]
Applying the mean value theorem to both components of $\tilde\Psi$ we get the system
\begin{equation}
  \label{sysfin}
  \left\{
\begin{aligned}
  0&= \bar{t}(t_2,w_2)-\bar{t}(t_1,w_1) = \left( 1 +O\left(\frac{1}{\underline{w}}\right)  \right)(t_2-t_1) + \left( \frac{2}{g} +O\left(\frac{1}{\underline{w}}\right)  \right)(w_2-w_1)\\
  0&= \bar{w}(t_2,w_2)-\bar{w}(t_1,w_1) = \left( 2\ddf(\Pto_\xi) +O\left(\frac{1}{\underline{w}}\right)  \right)(t_2-t_1) + \left(1+ \frac{4}{g}\ddf(\Pto_\xi) +O\left(\frac{1}{\underline{w}}\right)  \right)(w_2-w_1),
\end{aligned}
\right.
\end{equation}
where $\underline{w}=\min(w_1,w_2)$ and
\[
\Pto_\xi = \Pto \left((1-\xi)t_2-\xi t_1, (1-\xi)w_2-\xi w_1\right),
\]
for some $\xi\in [0,1]$. We conclude noting that system \eqref{sysfin} has the only solution
\[
t_2-t_1 = w_2-w_1 =0
\]
for $w_2,w_1$ large enough since, from a direct computation, the determinant of the associated matrix is of the form $1 +O\left( \frac{1}{\underline{w}}\right)$. This concludes the proof coming back to the variables $(t,v)$. 

\end{proof}

\begin{remark}
  Note that we cannot guarantee that if $(\to_0,\vo_0) \in \AA_{v_*}$ then $\Psi(\to_0,\vo_0) \in \AA_{v_*}$. This is reasonable, since the ball can slow down decreasing its velocity at every bounce.
  However, a bouncing motion is defined for $v\geq 0$.
\end{remark}

\begin{remark}
  From the physical point of view, the condition $\Psi^n(\to_0,\vo_0)\in\AA_{v_*} $ for every $n$, implies that we can only hit the ball when it is falling. To prove it, suppose that $\to_0 =0$ and let us see what happens at the first iterate. The time at which the ball reaches its maximum height is $t^{max}=\frac{\vo_0}{g}$. On the other hand, the first impact time $\Pto$ satisfies,
  \[
\Pto \geq \frac{2}{g}\vo_0 - \frac{4}{g}\norm{\df}= t^{max}\left(         2-\frac{4}{\vo_0}\norm{\df} \right) > t^{max},
\]
where the last inequality comes from $\vo_0\in\RR_{v_*}$ and $v_*>4\norm{\df}$.
  \end{remark}

The variables time-velocity $(t,v)$ introduced before are not symplectic, so that we change to the variables time-energy 
$(\to,\Eo)$ defined by

\[
 (\to,\Eo) = \left(\to,\frac{1}{2}\vo^2\right),
\]
obtaining the conjugated map
\[
\Phi :\AA_{\Eo_*}
\longrightarrow  \AA, \qquad \Eo_*=\frac{1}{2}v_*^2
\]
defined by
\begin{equation}\label{eq:unbe}
  \left\{
  \begin{split}
\Pto =  & \to + \frac 2g \sqrt{2\Eo}-\frac 2g f[\to,\Pto]+\frac 2g \df(\to)
\\
\PEo =  & \frac{1}{2}\left( \sqrt{2\Eo} - 2f[\to,\Pto] + \df (\Pto)+\df(\to) \right)^2,
\end{split}
\right.
\end{equation}
that by Lemma \ref{well_def} is a $C^{2}$-embedding.

We have the following
\begin{lemma}
  \label{lemma:exact}
  The map $\Phi$ is exact symplectic and twist in $\AA_{e_*}$.
  The generating function is given by
  \begin{equation}
    \label{genfun}
    \begin{split}
      h(\to,\Pto)= &  \frac{g^2}{24}(\Pto-\to)^3+\frac{g}{2}(f(\Pto)+f(\to))(\Pto-\to)-\frac{(f(\Pto)-f(\to))^2}{2(\Pto-\to)}\\
      & -g\int^{\Pto}_{\to}f(s)ds+\frac{1}{2}\int^{\Pto}_{\to}
      (\dot{f}(s))^2ds,
    \end{split}
  \end{equation}
  where
  \[
  \Pto-\to\in\Omega =\{ (\to,\Pto)\in\RR^2 : \Pto>T(\to)\}
  \]
  for a $C^{2}$ function $T:\RR\rightarrow \RR$ such that $T(t+1)=T(t)+1$ and $\frac{2}{g}(v_*-2\norm{\df})<T(t)-t<\frac{2}{g}(v_*+2\norm{\df})$.\\
  Moreover, $\Phi$ preserves and twists infinitely the upper end.
  
\end{lemma}

\begin{proof}
The proof basically comes from \cite{kunzeortega2}.
Inspired by \eqref{defomega}, consider the set
\[
\Omega=\left\{(\to,\Pto)\in\RR^2 : \Pto-\to > \frac{2}{g}\left( v_*- f[\to,\Pto]+ \df(\to)\right) \right\}
\]
 and the function $ h:\Omega \rightarrow \RR$ defined in \eqref{genfun}.\\
 Note that, by the implicit function theorem, the set $\Omega$ can be written as
\[
\Omega=\left\{(\to,\Pto)\in\RR^2 : \Pto>T(\to) \right\}
\]
 for a $C^{2}$ function $T:\RR\rightarrow \RR$ such that $T(t+1)=T(t)+1$ and $\frac{2}{g}(v_*-2\norm{\df})<T(t)-t<\frac{2}{g}(v_*+2\norm{\df})$. Moreover, $T(t)-t>0$ since $v_*>4\norm{\df}$.  

For every $(\to, \Pto)\in \Omega$, $h(\to+1,\Pto+1)=h(\to,\Pto)$ and 
\begin{equation}
\label{gen_interm}
  \left\{
\begin{split}
\partial_1 h(\to,\Pto)&=-\frac{1}{2}\left[\frac{g}{2}(\Pto -\to) +f[\to,\Pto]- \dot{f}(\to)    \right]^2 \\
\partial_2 h(\to,\Pto)&=\frac{1}{2}\left[\frac{g}{2}(\Pto -\to) -f[\to,\Pto]+ \dot{f}(\Pto)    \right]^2.
\end{split}
\right.
\end{equation}
Note that if $(\to,\Eo)\in\AA_{\Eo_*}$ then from \eqref{eq:unbe} we have $(\to,\Pto(\to,\Eo))\in \Omega$ so that 

  \begin{equation*}
  \left\{
\begin{split}
\partial_1 h(\to,\Pto(\to,\Eo)) &= -\frac{1}{2}\left[\frac{g}{2}(\Pto(\to,\Eo) -\to) +f[\to,\Pto(\to,\Eo)]- \dot{f}(\to)    \right]^2 \\
\partial_2 h(\to,\Pto(\to,\Eo)) &=   \frac{1}{2}\left[\frac{g}{2}(\Pto(\to,\Eo) -\to) -f[\to,\Pto(\to,\Eo)]+ \dot{f}(\Pto(\to,\Eo))    \right]^2,
\end{split}
\right.
\end{equation*}
and we can write

\begin{equation}
  \label{gen_imp}
\left\{
\begin{split}
  \partial_1 h(\to,\Pto(\to,\Eo))&=- \Eo
  \\
\partial_2 h(\to,\Pto(\to,\Eo))&=\PEo(\to,\Eo).
%
\end{split}
\right.
\end{equation}
Hence, the map $\Phi: \AA_{e_*}\rightarrow\AA$ can be expressed in the implicit form (\ref{gen_imp}).
On the other hand, in $\Omega$,
\[
\partial_{12} h(\to,\Pto) = -\left[\frac{g}{2}(\Pto -\to) +f[\to,\Pto]- \dot{f}(\to)    \right]\left(\frac{g}{2}+\partial_{\Pto} f[\to,\Pto]     \right)<0,   
\]
where the inequality comes from two computations made in the proof of Lemma \ref{well_def}. Hence the system
\begin{equation*}
\left\{
\begin{split}
\partial_1 h(\to,\Pto)&=-\Eo \\
\partial_2 h(\to,\Pto)&=\PEo
\end{split}
\right.
\end{equation*}
defines implicitly two functions $\PEo(\to,\Eo),\Pto(\to,\Eo) : \AA_{e_*}\rightarrow \RR$ that by \eqref{gen_interm} satisfy \eqref{eq:unbe}.

Relation (\ref{gen_imp}) also shows that in $\AA_{e_*}$
\[
\begin{split}
  d (h(\to,\Pto(\to,\Eo))) &= \partial_2 h(\to,\Pto(\to,\Eo))d\Pto(\to,\Eo) + \partial_1 h(\to,\Pto(\to,\Eo)) d\to \\
  &=\PEo(\to,\Eo)d\Pto(\to,\Eo)-\Eo d\to.
  \end{split}
\]
The twist condition comes from implicit differentiation in \eqref{gen_imp}:
\[
\frac{\partial\Pto}{\partial e}(\to,\Eo) =-\frac{1}{\partial_{12} h(\to,\Pto(\to,\Eo))} >0.
\]
Finally, from \eqref{eq:unbe}, uniformly in $t$,
\[
\Pto-\to\rightarrow +\infty \quad\mbox{and}\quad \PEo\rightarrow +\infty
\]
as $e\rightarrow +\infty$.
%
%
%
\end{proof}



We stated Theorems \ref{teomin} and \ref{teo_nec} for exact symplectic twist diffeomorphisms defined in the whole cylinder $\AA$. Hence, we need the following extension lemma (see for example \cite{maro3,maro2,maro1} and \cite[Theorem 8.1]{matherforni})
\begin{lemma}
  \label{extension}
There exists a $C^{2}$ exact symplectic and twist diffeomorphism $\tilde{\Phi}:\AA\rightarrow\AA$ such that $\tilde{\Phi} \equiv \Phi$ on $\AA_{e_*}$ and $\tilde{\Phi}\equiv\Phi_0$ on $\AA\setminus\AA_{\frac{e_*}{2}}$ where $\Phi_0$ is the integrable twist map $\Phi_0(t,e)=(t+e,e)$. Moreover, $\tilde{\Phi}$ preserves the ends of the cylinder and twists them infinitely.
\end{lemma}

We are going to prove that $\tilde{\Phi}$ has no invariant curves in $\AA_{e^*}$ for some $e^*>e_*$ large enough. To this aim, we need to bound the oscillations of possible invariant curves of $\tilde{\Phi}$ contained in $\AA_{\frac{e_*}{2}}$ and guarantee which ones are contained in $\AA_{e_*}$. Let us write $\tilde{\Phi}(t,e)=(\tilde{T}(t,e),\tilde{E}(t,e))$ and introduce  
\begin{equation}\label{maxtilde}
  \begin{split}
    E^* &:= \max \left\{ |\tilde{E}(t,e)-e| \: :\:(t,e)\in\TT\times\left[\frac{e_*}{2},e_*\right] \right\}, \\
    T^* &:= \max \left\{ \left|\tilde{T}(t,e)-t-\frac{2}{g}\sqrt{2\Eo}\right| \: :\:(t,e)\in\TT\times\left[\frac{e_*}{2},e_*\right] \right\}.
    \end{split}
\end{equation}

\begin{lemma}
\label{bound_curve}
  Let  $\Gamma\subset\AA_{\frac{e_*}{2}}$ be an invariant curve of ${\tPhi}$. Suppose that $\Gamma\cap\AA_{e^\sharp}\neq\emptyset$ for some $e^\sharp$ satisfying $\sqrt{2e^\sharp}>\sqrt{2e_*} +2\norm{\df}+g+\frac{g}{2}\max\left\{ \frac{4}{g}\norm{\df}, T^*\right\}$. Then, choosing $e^\flat$ satisfying $\sqrt{2e^\flat} =\sqrt{2\Eo^\sharp}-2\norm{\df}-g-\frac{g}{2}\max\left\{ \frac{4}{g}\norm{\df}, T^*\right\}$, we have 
  
  \[
\AA_{\Eo^\sharp}\subset\AA_{\Eo^\flat}\subset\AA_{\Eo_*},
  \]
and 
  \[
\Gamma\subset\AA_{\Eo^\flat}\subset\AA_{\Eo_*}.
  \]

  \end{lemma}
\begin{proof}
The first part comes directly from the hypothesis. Let us prove the second part.  Let $(t,e)\in\Gamma$ and denote $(\Pto,\PEo)=\tPhi(\to,\Eo)\in\Gamma$. It is known (see \cite{bangert}) that all the orbits on an invariant curve are minimal and have the same rotation number $\omega$. More precisely, there exists $\omega$ such that for every $(t,e)\in\Gamma$, we have
  \[
|\Pto - \to -\omega| < 1.
  \]
On the other hand, from \eqref{eq:unbe} and \eqref{maxtilde} for every $(\to,\Eo)\in\AA_{\frac{e_*}{2}}$ 
  \begin{equation*}
  \left|\Pto - \to-\frac{2}{g}\sqrt{2\Eo} \right|\leq\max\left\{ \frac{4}{g}\norm{\df}, T^*\right\}.
  \end{equation*}
  We deduce that
  \begin{equation}
     \label{prima_est}
\frac{2}{g}\sqrt{2\Eo}>\omega-1- \max\left\{ \frac{4}{g}\norm{\df}, T^*\right\}  \quad\mbox{for every } (\to,\Eo)\in\Gamma.
  \end{equation}
  By hypothesis, there exists a point $(\to^+,\Eo^+)\in\Gamma\cap\AA_{e^\sharp} $. As before we have that
 \[
|\Pto^+ - \to^+ -\omega| < 1, 
\]
and, since $\AA_{e^\sharp}\subset\AA_{e_*}$,
  \begin{equation*}
  \left|\Pto^+ - \to^+-\frac{2}{g}\sqrt{2\Eo^+} \right|\leq \frac{4}{g}\norm{\df},
  \end{equation*}   
  from which
  \[
\omega >\frac{2}{g}\sqrt{2\Eo^+}- \frac{4}{g}\norm{\df}-1 >\frac{2}{g}\sqrt{2\Eo^\sharp}- \frac{4}{g}\norm{\df}-1.
  \]
  Plugging this last inequality into \eqref{prima_est} we get that if $(t,e)\in\Gamma$, then
  \[
  \sqrt{2\Eo} > \sqrt{2\Eo^\sharp}-2\norm{\df}-g-\frac{g}{2}\max\left\{ \frac{4}{g}\norm{\df}, T^*\right\} >\sqrt{2\Eo_*},
  \]
  that is
  \[
  \sqrt{2\Eo} > \sqrt{2\Eo^\flat} >\sqrt{2\Eo_*},
  \]
  from which, squaring, we get the thesis.

\end{proof}

Let us start computing the corresponding functions $a(t),b(t)$ defined in \eqref{defab}, and the constants $D^\pm$ from Theorem \ref{teo_nec}. Note that, from Lemma \ref{extension}, the diffeomorphism $\tPhi$ satisfies the hypothesis of Lemma \ref{lemfi}. Hence, given an invariant curve $\Gamma$, the homeomorphism $\varphi$ and the functions $a$ and $b$ are well defined. Moreover, we recall that by Birkhoff Theorem every point $(t,e)\in\Gamma$ is of the form $(t,P(t))$ for a Lipschitz 1-periodic function $P$.   
\begin{lemma}
\label{tech}
Let $\Gamma\subset\AA_{e_*}$ be an invariant curve of $\tPhi$. We have, for every $(t,e)\in\Gamma$,
\begin{equation*}
  \begin{split}
    a(t) &= \sqrt{2\Eo}\left(g+2\ddot{f}(\to) +  R^A \right), \qquad |R^A| \leq \left(\frac{16(g+3\norm{\ddf})}{\sqrt{2}} \right)\frac{\norm{\df}}{\sqrt{\Eo}},
    \\
    b(t) & = \sqrt{2\Eo}\left(\frac{g}{2} + R^B   \right), \qquad |R^B| \leq \left(\frac{(7g+2\norm{\ddf})}{\sqrt{2}} \right)\frac{\norm{\df}}{\sqrt{\Eo}},
    \\
    b(\varphi(t)) & = \sqrt{2\Eo}\left(\frac{g}{2} + R^{\tilde{B}}   \right), \qquad |R^{\tilde{B}}| \leq \left(\frac{(5g+2\norm{\ddf})}{\sqrt{2}} \right)\frac{\norm{\df}}{\sqrt{\Eo}}.
  \end{split}
  \end{equation*}
Moreover, if we denote $M=\max\ddf$ and $\underline{\Eo}=\min\{ \Eo > \Eo_* \: :\: (t,e)\in\Gamma \}$ we can choose, as  $\underline{\Eo}\rightarrow\infty$
\begin{align*}
  D^+ &=\frac{ g+2M+2\sqrt{M^2 +gM} }{g} + O\left(\frac{1}{\sqrt{\underline{\Eo}}}\right), \\
  D^-& =\frac{ g+2M-2\sqrt{M^2 +gM} }{g} + O\left(\frac{1}{\sqrt{\underline{\Eo}}}\right).
 \end{align*}

\end{lemma}  
\begin{proof}
 Let us denote $\tPhi^n(t,e)=(\to_n,\Eo_n)$. Since $\Gamma\subset\AA_{e_*}$ we have that $\tPhi\equiv\Phi$ so that $(\to_n,\Eo_n)_{n\in\ZZ}$ is given by equations \eqref{eq:unbe} and we can consider the generating function $h$ defined in \eqref{genfun}. Moreover, from the definition of $\varphi$ in \eqref{defi}, we have $t_1=\varphi(t)$ and $t_{-1}=\varphi^{-1}(t)$.  
  
From \eqref{gen_interm} and using that
\[
\partial_{\to}f[\to,\to_1]=\frac{-\df(\to)+f[\to,\to_1]}{\to_1-\to}, \quad
\partial_{\Pto}f[\to,\to_1]=\frac{\df(\to_1)-f[\to,\to_1]}{\to_1-\to},
\]
we have
\begin{equation*}
  \begin{split}
    h_{11}(\to,\to_1)& = \frac{g}{2}(\to_1 -\to)\left(\frac{g}{2}+\ddot{f}(\to) \right) + \left(\partial_{\to}f[\to,\to_1]-\ddot{f}(\to)\right)\left(\dot{f}(\to) - f[\to,\to_1]\right),
    \\
    h_{22}(\to,\to_1)& = \frac{g}{2}(\to_1 -\to)\left(\frac{g}{2}+\ddot{f}(\to_1) \right) + \left(\partial_{\Pto}f[\to,\to_1]-\ddot{f}(\to_1)\right)\left( f[\to,\to_1] - \dot{f}(\to_1) \right),
     \\
    h_{12}(\to,\to_1)& = -\frac{g^2}{4}(\to_1 -\to) + \partial_{\Pto}f[\to,\to_1]\left(\dot{f}(\to) -f[\to,\to_1]  \right) -\frac{g}{2}  \left( 2f[\to,\to_1] + \dot{f}(\to_1)- \dot{f}(\to) \right).
  \end{split}
  \end{equation*}
These formulas give,
\begin{equation*}
  \begin{split}
    a(t) &= h_{22}(\to_{-1},\to)+h_{11}(\to,\to_{1}) = \frac{g}{2}(\to_{1} -\to_{-1})\left(\frac{g}{2}+\ddot{f}(\to) \right) + R^a, \\
    b(t) & =- h_{12}(\to_{-1},\to) = \frac{g^2}{4}(\to -\to_{-1})+R^b,
  \end{split}
  \end{equation*}
with
\[
|R^a| \leq 8\norm{\df}\norm{\ddf}, \quad  |R^b| \leq 2\norm{\df}(\norm{\ddf}+2g). 
\]
%
%
%
Moreover, from \eqref{eq:unbe} we have
\begin{align*}
  \to_{1}-\to_{-1} &= \frac{4}{g}\sqrt{2\Eo_{-1}}+\frac{4}{g}\left(\df(\to_{-1})+\df(\to)\right)-\frac{6}{g}f[\to_{-1},\to]-\frac{2}{g}f[\to,\to_{1}] \\
  & =\frac{4}{g}  \sqrt{2\Eo_{-1}}\left(1 + R^1 \right), \quad  |R^1| \leq \frac{4}{\sqrt{2\Eo_{-1}}}\norm{\df}, \\
   \to-\to_{-1} &= \frac{2}{g}\sqrt{ 2\Eo_{-1}}\left(1 + R^2 \right), \quad |R^2|\leq \frac{2}{\sqrt{2\Eo_{-1}}}\norm{\df}, 
\end{align*}
then, using also that
\[
\sqrt{2\Eo_{-1}}=\sqrt{2\Eo}\left(1+R^3) \right), \quad |R^3|\leq \frac{4}{\sqrt{2\Eo}}\norm{\df},
\]
we have
\begin{equation*}
  \begin{split}
    a(t)&= \sqrt{2\Eo}\left(g+2\ddot{f}(\to) + R^A  \right),  \quad |R^A|\leq \frac{16(g+3\norm{\ddf})}{\sqrt{2\Eo}}\norm{\df},  \\
    b(t) & = \sqrt{2\Eo}\left(\frac{g}{2} +  R^B  \right),  \quad |R^B|\leq \frac{(7g+2\norm{\ddf})}{\sqrt{2\Eo}}\norm{\df}, \\
    b(t_{1}) & = \sqrt{2\Eo}\left(\frac{g}{2} +  R^{\tilde{B}}  \right),  \quad |R^{\tilde{B}}|\leq \frac{(5g+2\norm{\ddf})}{\sqrt{2\Eo}}\norm{\df}. 
  \end{split}
  \end{equation*}
Now, to compute $D^\pm$ we first search for the expressions of $B^\pm$ and $C^\pm$. Since the curve $\Gamma$ is the graph of a Lipschitz (periodic) function, 
\[
|R^A|,|R^B|,|R^{\tilde{B}}| = O\left(\frac{1}{\sqrt{\underline{\Eo}}}\right) \quad\mbox{as } \underline{\Eo}\rightarrow\infty,
\]
and we have
\begin{align*}
 \frac{a(t)}{b(\varphi(t))}&= \frac{a(t)}{b(t_{1})}=\frac{2}{g}\left(g+2\ddf(t)\right)+O\left(\frac{1}{\sqrt{\underline{\Eo}}}\right) \leq\frac{2}{g}\left(g+2M\right)+O\left(\frac{1}{\sqrt{\underline{\Eo}}}\right), \\
  \frac{b(t)}{b(\varphi(t))}&= \frac{b(t)}{b(t_{1})}= 1+O\left(\frac{1}{\sqrt{\underline{\Eo}}}\right) ,
\end{align*}
and similar estimates hold for $ \frac{a(t)}{b(t)}$ and $\frac{b(\varphi(t))}{b(t)}$. Since these estimates are uniform on $\Gamma$, we can choose
\[
B^\pm =\frac{2}{g}\left(g+2M\right)+O\left(\frac{1}{\sqrt{\underline{\Eo}}}\right) , \qquad C^\pm= 1+O\left(\frac{1}{\sqrt{\underline{\Eo}}}\right),
\]
for which conditions \eqref{cond_pos} easily holds for $\underline{\Eo}\rightarrow \infty$ and the expressions of $D^\pm$ follow from a straightforward computation. 

\end{proof}

We are now ready to prove that Theorem \ref{teomin} gives the following

\begin{proposition}
  Let $m:=\min\ddf $ and suppose that $m < -\frac{g}{2}$. Then, if there exists an invariant curve $\Gamma\subset \AA_{\frac{e_*}{2}}$ of $\tPhi$ then,
  \[
 \Gamma\subset\AA_{\frac{e_*}{2}}\setminus\AA_{e^*}  \subset \AA_{\frac{e_*}{2}},
  \]
  with
  \[
\sqrt{2e^*} = \sqrt{2e_*}+ \norm{\df}\frac{16(g+3\norm{\ddf})}{-(g+2m)} + 2\norm{\df}+g+\frac{g}{2}\max\left\{ \frac{4}{g}\norm{\df}, T^*\right\}.
  \]

\end{proposition}

\begin{proof}
  Suppose that there exists an invariant curve such that $\Gamma\cap\AA_{e^*}\neq\emptyset$. Since
  \begin{equation*}
    \begin{split}
      \sqrt{2e^*}- 2\norm{\df}-g-\frac{g}{2}\max\left\{ \frac{4}{g}\norm{\df}, T^*\right\} &= \sqrt{2e_*}+ \norm{\df}\frac{16(g+3\norm{\ddf})}{-(g+2m)} > \sqrt{2e_*},
      \end{split}
  \end{equation*}
  we can apply Lemma \ref{bound_curve} and get that $\Gamma\subset\AA_{e^\flat}$ with
  \[
\sqrt{2e^\flat} = \sqrt{2e_*}+ \norm{\df}\frac{16(g+3\norm{\ddf})}{-(g+2m)}>\sqrt{2e_*}.
  \]
We show how to get a contradiction with Theorem \ref{teomin}. Let $t_0$ be such that $m = \ddf(t_0)$ and consider the corresponding point $(t_0,e_0)\in\Gamma$. By Lemma \ref{tech},
  \begin{equation*}
    a(t_0) = \sqrt{2\Eo_{0}}\left(g+2\ddot{f}(\to_0) +  R_0^A \right), \qquad |R_0^A| \leq 16(g+3\norm{\ddf})\frac{\norm{\df}}{\sqrt{2\Eo_{0}}}.
\end{equation*}
 Then
$a(t_0) = \sqrt{2\Eo_{0}}\left(g+2m +  R_0^A \right) $ is negative if $g+2m +  R_0^A<0$. This happens if $g+2m<0$ and
\begin{equation*}
|R_0^A|\leq 16(g+3\norm{\ddf})\frac{\norm{\df}}{\sqrt{2\Eo_{0}}} < -(g+2m),
\end{equation*}
that is 
\[
\sqrt{2e_0} >  \norm{\df}\frac{16(g+3\norm{\ddf})}{-(g+2m)}. 
\]
By Theorem \ref{teomin} there cannot exist  invariant curves crossing the vertical line $ \left\{(t_0,e)\: :\: \sqrt{2\Eo}>  \norm{\df}\frac{16(g+3\norm{\ddf})}{-(g+2m)} \right \} $. Since $\sqrt{2e^\flat}> \norm{\df}\frac{16(g+3\norm{\ddf})}{-(g+2m)}$, this is a contradiction with $\Gamma\subset\AA_{e^\flat}$ . 
  
\end{proof}

The main result of this section comes from an application of Theorem \ref{teo_nec}.

\begin{proposition}
\label{propmain}
  Suppose that
\[
m< -\frac{g}{   1+\sqrt{1 +\frac{g}{M}}  }.
\]
 Then there exists $e^*$, with $\sqrt{2e^*}>\sqrt{2e_*} +g+2\norm{\df}+\frac{g}{2}\max\left\{ \frac{4}{g}\norm{\df}, T^*\right\} $ such that if there exists an invariant curve $\Gamma\subset \AA_{\frac{e_*}{2}}$ of $\tPhi$ then,
  \[
 \Gamma\subset \AA_{\frac{e_*}{2}}\setminus\AA_{e^*}  \subset \AA_{\frac{e_*}{2}}.
  \]
\end{proposition}

\begin{proof}
  The proof goes as before applying Theorem \ref{teo_nec} instead of Theorem \ref{teomin}. Fix $t_0$ such that $m = \ddf(t_0)$. We claim (and prove later) that there exists $e^+$ satisfying
  \[
\sqrt{2e^+}>\sqrt{2e_*} +g+2\norm{\df}+\frac{g}{2}\max\left\{ \frac{4}{g}\norm{\df}, T^*\right\}
  \]
 and  such that there are no invariant curves crossing the vertical line $ \mathcal{L} =\{(t_0,e)\: :\: \sqrt{2\Eo}>\sqrt{2e^+}  \}$. This gives the thesis choosing $e^*$ such that $\sqrt{2e^*}>\sqrt{2e^+}+g+2\norm{\df}+\frac{g}{2}\max\left\{ \frac{4}{g}\norm{\df}, T^*\right\}   $. Actually, any invariant curve $\Gamma$ satisfying $\Gamma\cap\AA_{e^*}\neq\emptyset$ would satisfy, by Lemma \ref{bound_curve}, $\Gamma\subset\AA_{e^+}$ and then cross the line $\mathcal{L}$.

We now go back to the proof of the claim.  Fix an invariant curve $\Gamma\subset\AA_{\frac{e_*}{2}}$ and pick the orbit $(t_n,e_n)_{n\in\ZZ}$ through $(t_0,e_0)$. We show that for $e_0$ large
\begin{equation}
  \label{necc}
a(t_0)< b(t_{1})D^-+\frac{b(t_0)}{D^+},
\end{equation}
contradicting Theorem \ref{teo_nec}.
  By Lemma \ref{bound_curve}, denoting $\underline{\Eo}= \min \{\Eo \: :\: (\to,\Eo)\in\Gamma  \}$ we have
\begin{equation}
\label{bo1}
\sqrt{2\underline{\Eo}}>\sqrt{2e_0}-g-2\norm{\df}-\frac{g}{2}\max\left\{ \frac{4}{g}\norm{\df}, T^*\right\}.
\end{equation}
From Lemma \ref{tech} and using \eqref{bo1} we can write, as $\sqrt{2\underline{\Eo}}\rightarrow\infty$
\begin{align*}
  a(t_0) &= \sqrt{2\Eo_{0}}\left(g+2m + O\left(\frac{1}{\sqrt{\underline{\Eo}}}\right) \right), \\
  b(t_0),b(t_1) &=\sqrt{2\Eo_0}\left( \frac{g}{2}+ O\left(\frac{1}{\sqrt{\underline{\Eo}}}\right)   \right), \\
  D^-,\frac{1}{D^+} &= \left(\frac{ g+2M-2\sqrt{M^2 +gM} }{g} +O\left(\frac{1}{\sqrt{\underline{\Eo}}}\right)  \right). 
\end{align*}
Hence, \eqref{necc} is satisfied if
\[
g+2m < g+2M-2\sqrt{M^2 +gM}  +O\left(\frac{1}{\sqrt{\underline{\Eo}}}\right),  
\]
that is
\[
m<M-\sqrt{M^2 +gM}  +O\left(\frac{1}{\sqrt{\underline{\Eo}}}\right),
\]
or equivalently
\begin{equation}
  \label{fin}
m <- \frac{g}{   1+\sqrt{1 +\frac{g}{M}}  } +O\left(\frac{1}{\sqrt{\underline{\Eo}}}\right).
\end{equation}
Since by hypothesis
\[
m <- \frac{g}{   1+\sqrt{1 +\frac{g}{M}}  },
\]
condition \eqref{fin} is satisfied for $\underline{\Eo}$ (and therefore $e_0$) large enough.

\end{proof}

\section{Diffusive orbits and chaotic dynamics}\label{sec:chaos}

In this Section we describe how our main result, Theorem \ref{teo_main}, follows from Proposition \ref{propmain}. In the following we consider the cylinder $\AA_{e^*}$ coming from Proposition \ref{propmain}.


We have the following
\begin{lemma}
  \label{lemmabo}
  For every $K>0$ there exists an orbit $(\hat{\to}_n,\hat{\Eo}_n)_{n\in\ZZ}$ of $\tilde{\Phi}$ such that
  $\sup \hat{\Eo}_n -\inf \hat{\Eo}_n \geq K$. Moreover, $\frac{e_*}{2}<\inf \hat{\Eo}_n$. 
\end{lemma}
\begin{proof}
  Suppose by contradiction that there exists $K>0$   such that every orbit of $\tilde\Phi$   is such that $\sup_{n\in\ZZ^+}e_n- \inf_{n\in\ZZ^+}e_n < K$.
  Then for every initial condition $(t_0,e_0)\in\AA$ with $e_0$ below a level $y^->e^*$ satisfies, for every $n$, $\Eo_n < y^- +K:= y^+$. Therefore the second part of Birkhoff Theorem \ref{birk} gives the existence of an invariant curve for $\tilde{\Phi}$ in $\AA_{e^*}$  contradicting Proposition \ref{propmain}. The last part of the statement comes from the fact that the map $\tilde\Phi$ is the integrable map for $e<\frac{e_*}{2}$. 
\end{proof}

The orbit $(\hat{\to}_n,\hat{\Eo}_n)_{n\in\ZZ}$ of $\tPhi$ could not be an orbit of $\Phi$ since it may be not contained in $\AA_{e_*}$. However we have

\begin{lemma}\label{lemmabo2}
  
Consider a constant $E>2\norm{\df}\left(2\sqrt{2(e_*+E^*)}+4\norm{\df}\right)$ where $E^*$ was defined in \eqref{maxtilde}. There exist two integers $n^-<n^+$ such that  $|\hat{e}_{n^+} -\hat{e}_{n^-}|\geq E$ and $\hat{e}_n> e_*$ for every integer $n\in[n^-,n^+]$.
\end{lemma}
\begin{proof}
From Lemma \ref{lemmabo} with $K=E+E^*+\frac{e_*}{2}+\epsilon$ we have that there exist two integers  $n^1<n^2$ such that $|\hat{e}_{n^2} -\hat{e}_{n^1}|\geq E+E^*+\frac{e_*}{2}$ and $\hat{e}_n> \frac{e_*}{2}$ for every integer $n$.
Consider the case $\hat{e}_{n^2} >\hat{e}_{n^1}$: the case $\hat{e}_{n^2} <\hat{e}_{n^1}$ can be studied similarly.
 We claim that there exists $n^3$ such that $ \hat{e}_{n^3}>e_*$ and
 \[
 \hat{e}_{n^2}-\hat{e}_{n^3}>E.
 \]
This is clear if $\hat{e}_{n^1}>e_*$. To prove the claim in the other case, we remember that $\frac{e_*}{2}<\hat{e}_{n^1}<e_*$, then,
  \[
  \hat{e}_{n^2}>\hat{e}_{n^1}+E+E^*+\frac{e_*}{2}>e_*+E+E^*>e_*+E^*.
  \]
  This implies that to go from $\hat{e}_{n^1}$ to $\hat{e}_{n^2}$ we must cross the strip $\Sigma=\TT\times [e_*,e_*+E^*]$ that has width equal to $E^*$.
  This cannot be done in one iterate by the definition of $E^*$. So that there exists $n^3$ such that $\hat{e}_{n^3}\in[e_*,e_* +E^*]$. Finally,
  \[
  \hat{e}_{n^2}-\hat{e}_{n^3}>e_*+E +E^*-(e_*+E^*)=E.  
  \]
  This argument allows to conclude the proof. Suppose that there exists 
  $n^4\in[n^3,n^2]$ such that $ \frac{e_*}{2}<\hat{e}_{n^4}<e_*$. As before, there exists $n^5\in[n^4,n^2]$ such that $\hat{e}_{n^5}\in[e_*,e_* +E^*]$ and $\hat{e}_{n^2}-\hat{e}_{n^5}>E$. Repeating this argument we find $\hat{e}_{n^2-1}\in[e_*,e_* +E^*]$ and $\hat{e}_{n^2}-\hat{e}_{n^2-1}>E$. This is a contradiction with the hypothesis on $E$, since by \eqref{eq:unbe}, if $e<e_* +E^*$, then
  $|\PEo-\Eo|<4\sqrt{2(e_*+E^*)}\norm{\df}+8\norm{\df}^2 $.
  
\end{proof}
 Lemma \ref{lemmabo2} gives the proof of the first statement in Theorem \ref{teo_main}. One just has to consider the bouncing motion with initial condition $(t_0,v_0)=(\hat{t}_{n^-},\sqrt{2\hat{e}_{n^-}})$.

 Concerning chaotic dynamics, we first recall that the diffeomorphism  $\tilde{\Phi}$ satisfies the hypothesis of Aubry-Mather theory in \cite{bangert} (see also \cite{maro3}, \cite[Theorem 8.1]{matherforni}).\\
In particular, for every $\omega\in\RR\setminus\QQ$ there exists a compact $\tilde{\Phi}$-invariant set $M_\omega$ with rotation number $\omega$ that is either an invariant curve or a Cantor set. Using some ideas as in Lemma \ref{bound_curve} there exists $\omega^*$ sufficiently large such that for $\omega>\omega^*$, $M_\omega\subset\AA_{e^*}$. This implies that $M_\omega$ is made of orbits of the original map $\Phi$ and, from Proposition \ref{propmain}, is a Cantor set. 

In this setting, the following theorem by Forni \cite{forni} implies chaotic dynamics. Let us fix $\omega>\omega^*$ and denote by $\sigma_\omega$ the unique $\tilde{\Phi}$-invariant ergodic Borel probability measure supported on $M_\omega$.

\begin{theorem}
  Let $F$ be an exact symplectic $C^1$ twist diffeomorphism of the cylinder $\AA$ that does not admit any invariant curve of rotation number $\omega\in\RR\setminus\QQ$. Then there exists a $F$-invariant ergodic Borel probability measure $\mu_\omega$, of angular rotation number $\omega$, having positive metric entropy. Moreover, $\mu_\omega$ can be chosen arbitrarily close to $\sigma_\omega$ in the weak topology on the space of compactly supported Borel probability measures on $\AA$.  
\end{theorem}

We can apply this theorem to the extended map $\tilde{\Phi}$ and note that, as we showed, has no invariant curves in $\AA_{e^*}$ and in particular, none with rotation number $\omega>\omega^*$. Since $\omega>\omega^*$, the measure $\sigma_\omega$ is also $\Phi$-invariant. The measure $\mu_\omega$ with positive metric entropy can be chosen arbitrarily close to $\sigma_\omega$, hence we can have $\supp\mu_\omega$ close to $M_\omega$ and contained in $\AA_{e^*}$. In this way we have that $\mu_\omega$ is $\Phi$-invariant. From the variational principle for the topological entropy we get the thesis.

\section{Conclusions}\label{sec:conclusion}
We considered the model of a free falling ball bouncing elastically on a racket moving in the vertical direction according to a regular periodic function $f$. We were interested in the possibility of diffusive motions and chaotic dynamics. Both problems have been already investigated and an affirmative answer was given if $\norm{\df}$ was sufficiently large. In \cite{maro4} we showed that large values of $\norm{\df}$ were not necessary to have unbounded motions. In the present paper we provided a sufficient condition depending only on $\ddf$ giving diffusive and chaotic motions.

The proof is based on the breaking of invariant curves, using a converse KAM method. The method is based on a variational characterization of (segments of) orbits on invariant curves.  We stress that Theorem \ref{teo_nec} is not optimal as more stringent necessary conditions can be obtained considering longer orbit segments. We guess that a more stringent criterion would give a weaker condition than the one we obtained in Theorem \ref{teo_main}. However this would have implied much more complicated computations and the result would not have been optimal. Note that obtaining an optimal condition in this context is equivalent to finding an explicit necessary and sufficient condition for the existence of invariant curves. This problem seems to be very hard.     

\section*{Acknowledgements}
The author would like to thank the unknown referee for several valuable advice that significantly improved the final version of the paper.

\end{document}